\documentclass[reqno,dvips]{amsart}

\usepackage{mathrsfs}
\usepackage{amsmath,amssymb,amsaddr}
\usepackage{cite}
\usepackage{graphicx}

\usepackage{color}

\usepackage{comment}

\theoremstyle{plain}
\newtheorem{thm}{Theorem}[section]
\newtheorem{lem}[thm]{Lemma}
\newtheorem{cor}[thm]{Corollary}
\newtheorem{dfn}[thm]{Definition}
\newtheorem{prop}[thm]{Proposition}
\newtheorem{rmk}[thm]{Remark}

\newtheorem{que}[thm]{Question}

\def\M{\mathscr{M}}
\def\N{\mathscr{N}}
\def\S{\mathscr{S}}

\def\Rset{\mathbb{R}}
\def\Sset{\mathbb{S}}
\def\Tset{\mathbb{T}}

\def\M{\widetilde{\mathcal{M}}}
\def\N{\widetilde{\mathcal{N}}}
\def\A{\widetilde{\mathcal{A}}}

\def\I{\mathcal{I}}
\def\S{\mathcal{S}}

\def\m{\mathfrak{M}}

\def\E{\mathcal{E}}

\def\supp{\mathrm{supp}}

\def\epsilon{\varepsilon}

\makeatletter
 \@addtoreset{equation}{section}
\makeatother

\begin{document}


\title[]%
{Properties of action-minimizing sets and weak KAM solutions via Mather's averaging functions}

\author[S. Motonaga]{Shoya Motonaga}

\address{Research Organization of Science and Technology,
Ritsumeikan University, 1-1-1 Noji-higashi, Kusatsu, Shiga 525-8577, JAPAN}

\email{motonaga@fc.ritsumei.ac.jp}

\date{\today}

\begin{abstract}
We study properties of action-minimizing invariant sets for Tonelli Lagrangian and Hamiltonian systems and weak KAM solutions to the Hamilton-Jacobi equation in terms of Mather's averaging functions. Our principal discovery is that exposed points and extreme points of Mather's alpha function are closely related to disjoint properties and graph properties of the action-minimizing invariant sets, which is also related to $C^0$ integrability of the systems and the existence of smooth weak KAM solutions to the Hamilton-Jacobi equation.
\end{abstract}
\maketitle


\section{Introduction}\label{intro}
	For Tonelli Lagrangian and Tonelli Hamiltonian systems,
	Mather \cite{M91} and Ma\~{n}\'e \cite{M96} developed a theory of action-minimizing sets.
	Their studies were connected with weak KAM theory \cite{F08},
	which is a theory of viscosity solutions to the Hamilton-Jacobi equation,
	and both theories are still being actively studied.
	It is well-know that Mather's alpha and beta functions, 
	which are the minimum of the averaging action
	with cohomological and homological parameters respectively, play central roles in the theory of the action-minimizing sets.
	The relationship between properties of these functions and dynamical system properties
	is an interesting problem.
	For example, many researchers paid attention to the differentiability of beta function
	and one of the remarkable results due to Massart and Sorrentino \cite{M11} is that the regularity of beta function implies the $C^0$ integrability of the system when the base manifold is 2-dimensional torus.
	However, such problem is still far from being fully understood.
	Moreover, the relationship between these functions and solutions of the Hamilton-Jacobi equation
	is also an interesting question.

	In this paper, we investigate the properties of the action-minimizing sets such as Mather/Aubry/Ma\~{n}\'e sets and weak KAM solutions of the Hamilton-Jacobi equation in terms of Mather's averaging functions.
	In particular, some convexities such as exposed points and extreme points of Mather's alpha function characterize disjoint properties and graph properties of these invariant sets.
	Moreover, we discuss $C^0$ integrability of the systems and propose some related questions.
	We also discuss some relationships between these results and the existence of smooth weak KAM solutions to the Hamilton-Jacobi equation.
	
	Let $M$ be a compact and connected smooth $n$-dimensional Riemannian manifold without boundary. Let $L$ be a Tonelli Lagrangian and $H$ be the corresponding Tonelli Hamiltonian.
	We denote by $\M_c/ \A_c/ \N_c$ the Mather/Aubry/Ma\~{n}\'e sets for $c\in H^{1}(M;\Rset)$ respectively.
	For a homology class $h\in H_{1}(M;\Rset)$, $\M_{h}$ stands for the homological Mather set for $h$.
	Let $\beta:H_1(M;\Rset)\to\Rset$ be Mather's beta function and $\alpha:H^1(M;\Rset)\to\Rset$ be Mather's alpha function.  See Section~\ref{pre} for the precise definitions and details.
	We recall that a point $c\in H^1(M;\Rset)$ is said to be an exposed point of $\alpha$
	if for any $c'\in H^1(M;\Rset)\setminus\{c\}$ the function $\alpha$ is not affine on the segment joining $c$ and $c'$.
	A point $c\in H^1(M;\Rset)$ is said to be an extreme point of $\alpha$
	if for any segment $S\subset H^1(M;\Rset)$ containing $c$ in its interior the function $\alpha$ is not affine on $S$. Note that exposed points are extreme points but the converse is not true in general.
	Moreover, the function $\alpha$ is said to be strictly convex if each $c\in H^1(M;\Rset)$ is an extreme point of $\alpha$, and this condition is equivalent to that each $c\in H^1(M;\Rset)$ is an exposed point of $\alpha$.
	We now state our first main result.
	\begin{thm}\label{thm:exposed}
		For given $c\in H^1(M;\Rset)$, the following conditions are mutually equivalent:
		\begin{itemize}
			\item[(i)]	$\beta$ is differentiable at $h$ for all $h\in \partial\alpha(c)$.
			\item[(ii)]	$c$ is an exposed point of $\alpha$.
			\item[(iii)]	$\M_{c}\cap \M_{c'}=\emptyset$ holds for any $c'\in H^1(M;\Rset)\setminus\{c\}$.
			\item[(iv)]	$\A_{c}\cap \A_{c'}=\emptyset$ holds for any $c'\in H^1(M;\Rset)\setminus\{c\}$.
			\item[(v)]	$\N_{c}\cap \N_{c'}=\emptyset$ holds for any $c'\in H^1(M;\Rset)\setminus\{c\}$.
		\end{itemize}
	\end{thm}
	
	Theorem~\ref{thm:exposed} gives a characterization of disjoint properties of $\M_c/ \A_c/ \N_c$ for fixed $c\in H^1(M;\Rset)$.
	Note that, by the superlinearity of $\alpha$, there are infinitely many exposed points of $\alpha$.
	 Therefore, such disjoint properties of $\M_c/ \A_c/ \N_c$ hold for infinitely many $c\in H^1(M;\Rset)$.
	We also remark that the condition (i) of Theorem~\ref{thm:exposed} appears in 
	the studies of asymptotically isolated invariant Lipschitz Lagrangian graph 
	(\cite{R14}, Theorem 1.1) and of weak integrability (\cite{BS12}, Theorem 1.1).
	We obtain the following corollary from Theorem~\ref{thm:exposed}.
	\begin{cor}\label{cor:C1-str}
		The following statements are mutually equivalent:
		\begin{itemize}
			\item[(i)]	$\beta$ is differentiable everywhere.
			\item[(ii)]	$\beta$ is $C^1$.
			\item[(iii)]	$\alpha$ is not constant on any segment.
			\item[(iv)] $\alpha$ is strictly convex.
			\item[(v)]	For any $c\neq c'\in H^1(M;\Rset)$, $\M_{c}\cap \M_{c'}=\emptyset$.
			\item[(vi)]	For any $c\neq c'\in H^1(M;\Rset)$, $\A_{c}\cap \A_{c'}=\emptyset$.
			\item[(vii)]	For any $c\neq c'\in H^1(M;\Rset)$, 
					$\N_{c}\cap \N_{c'}=\emptyset$.
		\end{itemize}
	\end{cor}
	It is worthwhile to investigate additional properties of Mather's averaging functions.
	\begin{cor}\label{cor:C1-str-2}
		The following conditions are equivalent:
		\begin{itemize}
			\item[(i)] $\beta$ is differentiable everywhere and strictly convex.
			\item[(ii)]	$\alpha$ is differentiable everywhere and has no constant segment.
		\end{itemize}
		Moreover, in this case,
		\begin{itemize}
			\item[(1)]	$\alpha$ and $\beta$ are $C^1$ and strictly convex functions.
			\item[(2)]	$\nabla\alpha$ and $\nabla\beta$ are homeomorphisms
					which satisfy $\nabla \beta=(\nabla\alpha)^{-1}$.
			\item[(3)]	For each $h\in H_1(M;\Rset)$ and $c\in H^1(M;\Rset)$,
					$\M^h=\M_{\nabla\beta(h)}$ and $\M_c=\M^{\nabla\alpha(c)}$ hold.
			\item[(4)]	For $h\neq h'\in H_1(M;\Rset)$, $\M^h\cap\M^{h'}=\emptyset$.
					Similarly, properties (v),(vi) and (vii) in Corollary~\ref{cor:C1-str} hold.
		\end{itemize}
	\end{cor}
	\begin{que}\label{Q1}
		Is there a Tonelli Lagrangian which satisfies condition (i) and (ii) of  Corollary~\ref{cor:C1-str-2} other than the integrable Lagrangian $L_0=||v||^2/2$ ?
	\end{que}
	\begin{que}\label{Q2}
		Is there a Tonelli Lagrangian whose alpha function is strictly convex but not $C^1$ ?
	\end{que}
	We remark that Questions~\ref{Q1} and ~\ref{Q2} are closely related to the following questions proposed in \cite{A15} (the definition of $C^0$ integrability will be given later):
	\begin{itemize}
		\item[(Q1)] Does a $C^0$ integrable Tonelli Hamiltonian exist that is not $C^1$ integrable?
		\item[(Q2)] Can an invariant torus of a $C^0$ integrable Tonelli Hamiltonian flow carry two
		invariant measures that have not the same rotation number?
	\end{itemize}
	If the answer of Question~\ref{Q1} is positive then that of (Q1) is also positive in the case of 2-dimensional torus since $C^0$ integrablility of Tonelli Hamiltonian  in the such case is equivalent to $C^1$-regularity of beta function (see \cite{M11}) and it is also equivalent to strict convexity of alpha function by Corollary~\ref{cor:C1-str}. Moreover, if the answer of (Q2) is positive then that of Question~\ref{Q2} is also positive since  $C^1$-regularity of alpha function is equivalent to the condition that all the rotation vectors of invariant measures on an invariant torus of a $C^0$ integrable Tonelli Hamiltonian are the same (see Proposition~\ref{prop:rotation} in Section~\ref{pre}).	
	
	Next we turn to the study of $C^0$ integrability of Tonelli Hamiltonian systems.
	\begin{dfn}[$C^0$ integrability]\label{dfn:C0}
		A Tonelli Hamiltonian $H$ on $T^*M$ is said to be $C^0$ integrable
		if it admits a foliation of the phase space by disjoint Lipschitz invariant Lagrange graphs,
		one for each possible cohomology class in $H^1(M;\Rset)$.
	\end{dfn}
	According to \cite{A11}, the following question is well known from specialists:
	\begin{que}[\cite{A11}]\label{que:C0}
		When Mather's beta function is everywhere differentiable, is the Hamiltonian $C^0$ integrable?
	\end{que}
	\begin{rmk}\label{rmk:C0}
	A positive answer to this question in the case of 2-dimensional torus is given in \cite{M11}.
	\end{rmk}
	We will provide a new aspect of Question~\ref{que:C0} in the present paper.
	The following theorem plays a key role in our discussions:
	\begin{thm}\label{thm:ext-alpha}
		Let $c\in H^1(M;\Rset)$ and $\alpha$ be Mather's alpha function.
		\begin{itemize}
			\item[(i)] If $\pi(\A_c)=M$, then $c$ is an extreme point of $\alpha$.
			\item[(ii)] If $\pi(\M_c)=M$ and $\alpha$ is differentiable at $c$, then $c$ is an exposed point of $\alpha$.
		\end{itemize}
	\end{thm}
	\begin{rmk}\label{rmk:pend}\ 
		\begin{itemize}
		\item[(1)] Note that $\pi(\A_c)=M$ holds if and only if the Hamilton-Jacobi equation $H(x, c+du)=\alpha(c)$ has a unique $C^1$ negative (resp. positive) weak KAM solution up to constants for $c=[\eta] \in H^1(M;\Rset)$ (see Proposition~\ref{prop:HJ} in Section~\ref{pre}).
		\item[(2)] There is an example which satisfies $\pi(\A_c)=M$
		and $c$ is a differentiable but not exposed point of Mather's alpha function
		(see Section~\ref{subsec:1-degree}).
		\end{itemize}
	\end{rmk}
	We state our result on $C^0$ integrability.
	\begin{thm}\label{thm:C0}
		We have:
		\begin{itemize}
			\item[(i)] If $\dim H^1(M;\Rset)\neq \dim M$, then the Hamiltonian $H$ cannot be $C^0$-integrable.
			\item[(ii)] If $\dim H^1(M;\Rset)= \dim M$, the followings are equivalent:
				\begin{itemize}
					\item[(1)]	The Hamiltonian $H$ is $C^0$-integrable.
					\item[(2)]	$\pi(\A_c)=M$ for all $c\in H^1(M;\Rset)$.
				\end{itemize}
		\end{itemize}
	\end{thm}
	From Corollary~\ref{cor:C1-str} and Theorem~\ref{thm:C0}, assuming $\dim H^{1}(M;\Rset)=\dim M$, we see that Question~\ref{que:C0}
	is equivalent to the following Question~\ref{que:HJ}.
	\begin{que}\label{que:HJ}
		When Mather's alpha function is strictly convex, does $\pi(\A_c)=M$ holds for all $c\in H^1(M;\Rset)$?
	\end{que}
	\begin{rmk}\label{rmk:HJ}\ 
	\begin{itemize}
		\item[(1)] From Remark~\ref{rmk:C0} and the equivalence of Questions~\ref{que:C0} and ~\ref{que:HJ}, the answer of Question~\ref{que:HJ} in the case of 2-dimensional torus is positive.
		\item[(2)] The answer of the converse of Question~\ref{que:HJ} is positive by Theorem~\ref{thm:ext-alpha}.
	\end{itemize}
	\end{rmk}
	Moreover, Theorem~\ref{thm:ext-alpha} and Remark~\ref{rmk:pend} (1) lead to the following questions:
	\begin{que}\label{que:HJ2}
		When $c=[\eta] \in H^1(M;\Rset)$ is an extreme or exposed point of Mather's alpha function $\alpha$,
		does the Hamilton-Jacobi equation $H(x, \eta+d_x u)=\alpha(c)$ have a classical solution?
		Moreover, does such convexity of $\alpha$ at $c$ determine regularities of the weak KAM solutions?
	\end{que}
	In Section~\ref{sec:examples}, we investigate two examples: single-degree-of-freedom mechanical Hamiltonians and KAM tori. In these examples, we see that an extreme (resp. exposed) point of Mather's alpha function corresponds to the condition that the projected Aubry (resp. Mather) set is the whole base manifold, which also corresponds to the existence of $C^1$ (resp. $C^2$) solutions to the Hamilton-Jacobi equation. We remark again that Mather's alpha function has  infinitely many exposed (thus extreme) points by its superlinearity.

\section{Prerequisites on action minimizing invariant sets}\label{pre}
In this section, we provide prerequisites on Aubry-Mather theory for Tonelli Lagrangians and Hamiltonians. See \cite{M91,M96,F08,S15} for more details. As stated in Section~\ref{intro}, let $M$ be a compact and connected smooth $n$-dimensional Riemannian manifold without boundary.
We denote a point of the tangent bundle $TM$ and the cotangent bundle $T^*M$ respectively by $(x,v)$ and $(x,p)$ with $v\in T_xM$ and $p\in T_x^*M$ for $x\in M$.
Recall that a \emph{Tonelli Lagrangian} $L:TM\to\Rset$ is a $C^2$ function which is strictly $C^2$-convex and superlinear in the fiber, i.e., $\partial^2 L/\partial v^2$ is positive definite and $\lim_{||v||\to +\infty} L(x,v)/||v||=+\infty$.
The Euler-Lagrange flow of the Tonelli Lagrangian $L$ is denoted by $\Phi_t^L$.
The corresponding Hamiltonian (the Legendre transformation of $L$), called \emph{Tonelli Hamiltonian}, is denoted by $H$, which has the same regularity as $L$ and also satisfies strictly $C^2$-convexity and superlinearity in the fiber.

\subsection{Action-minimizing measures and the Mather sets}
Let $\m(L)$ be the set of compactly supported $\Phi_t^L$-invariant probability measures on $TM$.
We consider the average action $A_L: \m(L)\to \Rset$ given by
\begin{align*}
	A_L(\mu)=\int_{TM} L \ d\mu.
\end{align*}
For a closed 1-form $\eta$ on $M$, it is well-known that a modified Tonelli Lagrangian given by $L_{\eta}(x,v)=L(x,v)-\langle\eta(x),v\rangle$ has the same Euler-Lagrange flow as $L$, but may have different average action when $\eta$ is not exact.
This leads to the definition of \emph{Mather's alpha function}
\begin{align*}
	\alpha(c)=-\min_{\mu\in\m(L)} \int L_\eta \ d\mu,\quad c=[\eta]\in H^1(M;\Rset).
\end{align*}
There exists a surjective map $\rho:\m(L)\to H_1(M;\Rset)$ given by
\begin{align*}
	\int_{TM} \langle\eta(x),v\rangle\ d\mu=\langle c,\rho(\mu)\rangle, \quad \forall c=[\eta]\in H^1(M;\Rset),
\end{align*}
which determines homological positions in $M$ of the supports of the invariant probability measures
and the image $\rho(\mu)$ is called the \emph{rotation vector} of $\mu$.
We consider the minimum average action under the constraint on the rotation vector and the minimum value of such problem is given by
\[
	\beta(h)=\min_{\mu\in{\rho^{-1}(h)}} \int L \ d\mu,
\]
which is known as \emph{Mather's beta function}. 
Both of $\alpha$ and $\beta$ are convex conjugate, i.e., they are convex functions such that
\[
	\alpha(c)=\sup_{h\in H_1(M;\Rset)} \{\langle c,h\rangle -\beta(h)\},\quad
	\beta(h)=\sup_{c\in H^1(M;\Rset)} \{\langle c,h\rangle -\alpha(c)\}.
\]
Note that $\alpha$ and $\beta$ have superlinear growth.
Since $\alpha$ and $\beta$ are convex, we can define the subdifferentials of these functions.
For $c\in H^1(M;\Rset)$, the \emph{subdifferential} of the convex function $\alpha$ at $c$, which is denoted by $\partial\alpha(c)$, is the set of points $h\in H_1(M;\Rset)$ such that
\begin{align*}
	\alpha(c')-\alpha(c)\ge \langle c-c', h\rangle,\quad \forall c'\in H^1(M;\Rset).
\end{align*}
Similarly, for $h\in H_1(M;\Rset)$ the subdifferential $\partial\beta(h)$ of the convex function $\beta$ at $h$ is the set of points $c\in H^1(M;\Rset)$ such that
\begin{align*}
	\beta(h')-\beta(h)\ge \langle c, h'-h\rangle,\quad \forall h'\in H_1(M;\Rset).
\end{align*}
It is well-known that
\[
h\in\partial\alpha(c) \quad\Leftrightarrow\quad
c\in\partial\beta(h) \quad\Leftrightarrow\quad
\alpha(c)+\beta(h)=\langle c,h\rangle
\]
for $c\in H^1(M;\Rset)$ and $h\in H_1(M;\Rset)$ (see Theorem 23.5 in \cite{R70}).
From the general results of convex analysis, we can easily obtain the following.
\begin{prop}\label{prop:convex}\
		\begin{itemize}
			\item[(i)]	$\beta$ and $\alpha$ are locally Lipschitz continuous, 
					thus differentiable almost everywhere.
			\item[(ii)]	For each $h\in H_1(M;\Rset)$ and $c\in H^1(M;\Rset)$,
					$\partial \beta(h)$ and $\partial \alpha(c)$ are nonempty, compact and convex sets.
			\item[(iii)]	If $\beta$ (resp. $\alpha$) is differentiable on some open convex set $U$,
					then $\beta$ (resp. $\alpha$) is $C^1$ on $U$.
			\item[(iv)]	$\beta$ (resp. $\alpha$) is strictly convex if and only if
					$\alpha$ (resp. $\beta$) is $C^1$.
			\item[(v)]	$\partial\beta$ (resp. $\partial\alpha$) is a one-to-one mapping if and only if $\beta$ (resp. $\alpha$) is strictly convex and differentiable everywhere.
		\end{itemize}
	\end{prop}
	\begin{proof}
	By Theorem 10.6 in \cite{R70} and Rademacher's theorem, we obtain (i).
	Using Theorem 24.7, 25.5.1, 26.3, 26.3.1 in \cite{R70}, we also have (ii), (iii), (iv) and (v).
	Note that $\alpha$ and $\beta$ are finite and $\alpha^*=\beta$.
	\end{proof}

Let $\m_c$ (resp. $\m^h$) be the set of invariant probability measures which attain the minimum average action of $L_\eta$ (resp. $L$) in $\m(L)$ (resp. $\rho^{-1}(h)$) for $c=[\eta]\in H^1(M;\Rset)$ (resp. $h\in H_1(M;\Rset)$):
\begin{align*}
	\m_c=\{\mu\in \m(L); A_{L_\eta}(\mu)=-\alpha(c)\}, \quad
	\m^h=\{\mu\in \rho^{-1}(h); A_{L}(\mu)=\beta(h)\}.
\end{align*}
We remark that $\m_c$ and $\m^h$ are not empty for each $c\in H^1(M;\Rset)$ and $h\in H_1(M;\Rset)$.
\begin{prop}\label{prop:rotation}
		For each $c=[\eta] \in H^1(M;\Rset)$, we have
		\begin{align}
			\partial \alpha (c)= \rho(\m_c).
		\end{align}
		In particular, $\alpha$ is differentiable at $c$
					if and only if all rotation vectors of $\m_c$ are the same.
	\end{prop}
	\begin{proof}
		Fix $h\in \rho(\m_c)$. This implies that there exists $\mu\in\m_c$ which satisfies
		$\rho(\mu)=h$ and $A_{L_\eta}(\mu)=-\alpha(c)$.
		Then
		\[
			A_L(\mu)=A_{L_\eta}(\mu)+\langle c,\rho(\mu)\rangle
			=\langle c,h\rangle -\alpha(c)\le 
			\sup_{c'\in H^1(M;\Rset)} \big\{\langle c',h\rangle -\alpha(c') \big\} = \beta(h)
		\]
		holds.
		On the other hand, by the definition of $\beta(h)$,
		we have $\beta(h)\le A_L(\mu)$.
		So the identity
		\[
			\alpha(c)+\beta(h)=\langle c,h\rangle
		\]
		holds and it implies $h\in\partial \alpha(c)$.
		
		Next, take $h\in\partial \alpha(c)$ and $\mu\in \m^h$.
		Note that $\rho(\mu)=h$ and $\beta(h)=A_L(\mu)$ hold.
		Then it follows from $h\in\partial \alpha(c)$ that
		\[
			-\alpha(c)=\beta(h)-\langle c,h\rangle=A_L(\mu)-\langle c,\rho(\mu)\rangle
			=A_{L_\eta}(\mu).
		\]
		So $\mu$ is an element of $\m_c$. Thus, we obtain $h=\rho(\mu)\in\rho(\m_c)$.
	\end{proof}
	We now define the action-minimizing sets due to Mather.
	For $c\in H^1(M;\Rset)$ and $h\in H_1(M;\Rset)$,
	the \emph{Mather set} of the cohomology class $c$ and the homology class $h$ are given by
	\begin{align*}
		\M_c=\bigcup_{\mu\in \m_c} {\rm supp}\  \mu,\quad
		\M^h=\bigcup_{\mu\in \m^h} {\rm supp} \ \mu,
	\end{align*}
	respectively. They are nonempty, compact and invariant sets for $\Phi^L_t$.

\subsection{Action-minimizing curves and the Aubry set and the Ma\~{n}\'e set}
	Next, we introduce action-minimizing curves and associated invariant sets. For given $x,y\in M$ and $T>0$, let $\mathcal{C}^T(x,y)$ be the set of absolutely continuous curves $\gamma:[0,T]\to M$ with $\gamma(0)=x$ and $\gamma(T)=y$. The associated action along $\gamma\in \mathcal{C}^T(x,y)$ is defined by
	\begin{align*}
		\mathscr{A}_L(\gamma)=\int_0^T L(\gamma(t),\dot{\gamma}(t)) \ dt.
	\end{align*}
	For a closed 1-form $\eta$, the \emph{Ma\~{n}\'e potential} $\phi_\eta:M\times M\to \Rset$ is given by
	\begin{align*}
		\phi_\eta(x,y)=\inf_{T>0}\min_{\gamma\in\mathcal{C}^T(x,y)} \mathscr{A}_{L+\alpha(c)}(\gamma)
	\end{align*}
	where $c$ is the cohomology class of $\eta$. Note that $\phi_\eta$ is Lipschitz and satisfies
	\begin{align}\label{eqn:mane}
		\phi_\eta(x,y)+\phi_\eta(y,x)\ge 0, \quad x,y\in M.
	\end{align}
	
	Let $c=[\eta]\in H^1(M;\Rset)$. An absolutely continuous curve $\gamma:\Rset\to M$ is said to be \emph{$c$-semistatic} if
	\begin{align*}
		\mathscr{A}_{L+\alpha(c)}(\gamma|_{[a,b]})=\phi_{\eta}(\gamma(a),\gamma(b))
	\end{align*}
	holds for any $[a,b]\subset \Rset$. Note that the above definition is well-defined. We remark that a $c$-semistatic curve is always a solution of the Euler-Lagrange equation.
	An absolutely continuous curve $\gamma:\Rset\to M$ is said to be \emph{$c$-static} if
	\begin{align*}
		\mathscr{A}_{L+\alpha(c)}(\gamma|_{[a,b]})=-\phi_{\eta}(\gamma(b),\gamma(a))
	\end{align*}
	holds for any $[a,b]\subset \Rset$. It is clear that a $c$-static curve is $c$-semistatic.
	Then the \emph{Ma\~{n}\'e set} $\N_c$ and the \emph{Aubry set} $\A_c$ for $c\in H^1(M;\Rset)$ are given by
	\begin{align*}
		\N_c&=\{(\gamma(t),\dot{\gamma}(t))\in TM; \gamma \text{ is $c$-semistatic}\},\\
		\A_c&=\{(\gamma(t),\dot{\gamma}(t))\in TM; \gamma \text{ is $c$-static}\},
	\end{align*}
	which are nonempty, compact and invariant sets for $\Phi^L_t$.

	For the action-minimizing sets $\M_c$, $\A_c$ and $\N_c$ of cohomology class $c\in H^1(M; \Rset)$,
	the following important inclusions hold:
	\begin{align}\label{eqn:inclusions}
		\M_c\subset \A_c\subset \N_c\subset \E(\alpha(c))
	\end{align}
	where $\E(\alpha(c))=\{(x,v)\in TM: E(x,v)=\alpha(c)\}$ is the level set of the energy $E(x,v)=\langle \frac{\partial L}{\partial v}(x,v), v\rangle-L(x,v)$.
	Another important feature is that $\M_c$ and $\A_c$ are graphs over $M$, i.e., the projections along the fibers $\pi|_{\M_c}$ and $\pi|_{\A_c}$ are injective with Lipschitz inverses
	where $\pi:TM\to M$ is the canonical projection.
	The set $\pi(\M_c)$ and $\pi(\A_c)$ are called the \emph{projected Mather set} and the \emph{projected Aubry set} for $c\in H^1(M;\Rset)$.

	\subsection{Weak KAM solutions for the Hamilton-Jacobi equation}
	Fathi's weak KAM theory \cite{F08} provides interesting characterizations of the Aubry set and the Ma\~{n}\'{e} set in terms of weak KAM solutions of Hamilton-Jacobi equation. In this subsection, we would like to move the Hamiltonian formalism rather than the Lagrangian one. Thus we consider the associated Hamiltonian $H$ with the Lagrangian $L$. The Hamiltonian $H$ induces the Hamiltonian flow $\Phi^H_t$, which is conjugate to the Euler-Lagrange flow $\Phi^L_t$ through the Legendre transform $\mathcal{L}$ associated with $L$.
	 For $c\in H^1(M;\Rset)$ we can define the dual Mather/Aubry/Ma\~{n}\'e sets
	 in the cotangent bundle $T^*M$ as
	 \begin{align*}
	 	\M_c^*=\mathcal{L}(\M_c),\quad
		\A_c^*=\mathcal{L}(\A_c), \quad
		\N_c^*=\mathcal{L}(\N_c).
	 \end{align*}
	Note that the inclusions among these sets and the graph properties of the Mather set and the Aubry set still hold. Abusing the notation we denote by $\pi$ the canonical projection $\pi:T^*M\to M$.
	
	Let $\eta$ be a closed 1-form whose cohomology class is $c=[\eta]\in H^1(M;\Rset)$.
	Consider the Hamilton-Jacobi equation of the form
	\begin{align}\label{eqn:HJ}
		H(x,\eta(x)+d_x u)=\alpha(c).
	\end{align}
	It is generally impossible to find classical solutions of \eqref{eqn:HJ}
	and thus one consider a more weaker notion of solutions: weak KAM solutions.
	Following Fathi \cite{F08}, we introduce the negative (resp. positive) Lax-Oleinik semi-group $\{T^{\eta,-}_t\}_{t>0}$ (resp. $\{T^{\eta,+}_t\}_{t>0}$) on $u\in C(M;\Rset)$ defined by
	\begin{align*}
		T^{\eta,-}_t u(x)&=\min\left(u(\gamma(0))+\int_0^t L_{\eta} (\gamma(s),\dot{\gamma}(s))\ ds\right),\\
		T^{\eta,+}_t u(x)&=\max\left(u(\gamma(t))-\int_0^t L_{\eta} (\gamma(s),\dot{\gamma}(s))\ ds\right)
	\end{align*}
	where the minimum (resp. maximum) is taken on the set of absolutely continuous curves $\gamma:[0,t]\to M$ such that $\gamma(t)=x$ (resp. $\gamma(0)=x$).
	Fathi's weak KAM theorem asserts that for each closed 1-form $\eta$ with $c=[\eta]\in H^1(M;\Rset)$ there exists $u\in C^0(M;\Rset)$
	such that $T^{\eta,-}_t u=u-\alpha(c)t$ (resp. $T^{\eta,+}_t u=u+\alpha(c)t$) for all $t>0$.
	Such function $u$ is called a negative (resp. positive) \emph{weak KAM solution} for $L_\eta$ and
	the set of negative (resp. positive) weak KAM solutions for $L_\eta$ will be denoted
	by $\S^-_\eta$ (resp. $\S^+_\eta$).
	Two solutions $u_-\in \S^-_\eta$ and $u_+\in \S^+_\eta$ are said to be \emph{conjugate}
	if they coincide on $\pi(\M^*_c)$.
	Each negative weak KAM solution has a unique conjugate positive weak KAM solution.
	We will denote by $(u_-,u_+)$ a couple of conjugate weak KAM solutions
	and for each $(u_-,u_+)$ we define
	\begin{align*}
		\I(u_-,u_+)&=\{x\in M; \ u_-(x)=u_+(x)\},\\
		\widetilde{\I}(u_-,u_+)&=\{(x,\eta(x)+p)\in T^*M; x\in \I(u_-,u_+), p=d_xu_-=d_xu_+\}.
	\end{align*}
	Then for $c=[\eta]\in H^1(M;\Rset)$
	the Aubry set and the Ma\~{n}\'{e} set are characterized by weak KAM solutions:
	\begin{align}
		\A^*_c=\bigcap_{(u_-,u_+)} \I(u_-,u_+),\label{eqn:Fathi1}\\
		\N^*_c=\bigcup_{(u_-,u_+)} \I(u_-,u_+)\label{eqn:Fathi2}
	\end{align}
	where the intersection and the union are taken over all pairs $(u_-,u_+)$
	of conjugate weak KAM solutions for $L_\eta$.
	\begin{rmk}\label{rmk:Aubry}
		As stated in \cite{M11}, the dual Aubry set of cohomology class $c$ can be seen as
		the intersection of Lipschitz Lagrangian graphs in $T^*M$ with the standard symplectic form.
	\end{rmk}
	
	\begin{prop}\label{prop:HJ}
		For each $c=[\eta]\in H^1(M;\Rset)$, the following statements are equivalent:
		\begin{itemize}
			\item[(i)] $\pi(\A_c)=M$ holds.
			\item[(ii)] There is a unique $C^1$ negative (resp. positive) weak KAM solution $u$ for $L_\eta$.
		\end{itemize}
		Moreover, in this case, $u\in C^{1,1}(M;\Rset)\cap \S^-_{\eta}\cap \S^+_{\eta}$ and $u$ is a $C^1$ solution of the Hamilton-Jacobi equation \eqref{eqn:HJ}.
	\end{prop}
	\begin{proof}
	 (i)$\Rightarrow$(ii): Assume that (i) holds. Then $\A^*_c$ is a graph on $M$.
	 For any conjugate pair $(u_-,u_+) \in S^-_\eta\times S^+_\eta$,
	 it follows from Eq.~\eqref{eqn:Fathi1} that $\A^*=\{(x,\eta+du_-); x\in M\}=\{(x,\eta+du_+); x\in M\}$, 
	 which implies $du_+=du_-$ and
	 $L_\eta$ has a unique negative (resp. positive) weak KAM solution up to constants.
	 Since the pair $(u_-,u_+)$ is conjugate, we have $u_-=u_+\in S^-_\eta\cap S^+_\eta$.
	 We see that $u_-=u_+$ is a $C^1$ solution of
	 the Hamilton-Jacobi equation $H(x,\eta+du)=k$ for some $k\in\Rset$ by Theorem~4.11.8 in \cite{F08}
	 and the value $k$ must be $\alpha(c)$ by Corollary 4.8.5 in \cite{F08}.
	 Moreover, we obtain $u_-=u_+\in C^{1,1}(M;\Rset)$
	 since $\A^*=\{(x,\eta+du_\pm); x\in M\}$ holds and $\pi|_{\A_c}$ has the Lipschitz inverse.
	Consequently, we deduce that
	$L_\eta$ has a unique $C^1$ negative (resp. positive) weak KAM solution $u$
	which satisfies $u\in C^{1,1}(M;\Rset)\cap \S^-_{\eta}\cap \S^+_{\eta}$ and
	it is a $C^1$ solution of the Hamilton-Jacobi equation \eqref{eqn:HJ}.\\
	(ii)$\Rightarrow$(i): Let $u$ be a unique negative (resp. positive) weak KAM solution for $L_\eta$ which is $C^1$.
	Then $u\in \S^-_\eta\cap \S^+_\eta$ holds by Theorem~4.11.8 in \cite{F08}.
	Hence $(u,u)$ is the unique conjugate pair up to constants.
	Therefore, $\pi(\A_c)=M$ follows from Eq.~\eqref{eqn:Fathi1}.
	\end{proof}

\section{Proofs of Theorem~\ref{thm:exposed} and its corollaries}
	In this section, we show Theorem~\ref{thm:exposed} and Corollaries \ref{cor:C1-str} and \ref{cor:C1-str-2}. We begin with the following proposition.	
	\begin{prop}\label{prop:main}
		Let $\alpha:H^1(M;\Rset)\to\Rset$ be Mather's alpha function.
		Let $c_0, c_1\in H^1(M;\Rset)$ be two distinct cohomologies
		and let denote the segment $[c_0c_1]=\{c_t:=(1-t)c_0+tc_1; t\in [0,1]\}$.
		Then the following statements are mutually equivalent:
		\begin{itemize}
			\item[(i)] $\alpha$ is constant on $[c_0c_1]$.
			\item[(ii)] $\alpha$ is affine on $[c_0c_1]$.
			\item[(iii)] $\M_{c_t}\subset \M_{c_0}\cap \M_{c_1}$ holds for all $t\in (0,1)$.
			\item[(iv)] $\M_{c_0}\cap \M_{c_1}\neq \emptyset$ .
			\item[(v)] $\A_{c_t}\subset \A_{c_0}\cap \A_{c_1}$ holds for all $t\in (0,1)$.
			\item[(vi)] $\A_{c_0}\cap \A_{c_1} \neq \emptyset$.
		\end{itemize}
		Moreover, in this case,
		\begin{itemize}
			\item[(1)] $\pi(\M_{c_t})\subset \pi(\M_{c_0})\cap \pi(\M_{c_1})$ holds for all $t\in (0,1)$.
			\item[(2)] $\pi(\A_{c_t})\subset \pi(\A_{c_0})\cap \pi(\A_{c_1})$ holds for all $t\in (0,1)$.
		\end{itemize}
	\end{prop}
	\begin{proof}
		(i)$\Rightarrow$(ii), (iii)$\Rightarrow$(iv), (iv)$\Rightarrow$(vi),
		(v)$\Rightarrow$(vi), (iii)$\Rightarrow$(1) and (v)$\Rightarrow$(2) are trivial.

		(ii)$\Rightarrow$(iii):
		Assume that
		\begin{align*}
			\alpha(c_t)=(1-t)\alpha(c_0)+t\alpha(c_1), \quad t\in[0,1].
		\end{align*}
		This implies
		\begin{align*}
			(1-t)\int L_{\eta_{c_0}}d\mu_{c_t}+t\int L_{\eta_{c_1}}d\mu_{c_t}
			=(1-t)\min_{\mu\in \m} \int L_{\eta_{c_0}} d\mu
				+t\min_{\mu\in \m} \int L_{\eta_{c_1}} d\mu,
		\end{align*}
		equivalently,
		\begin{align*}
			(1-t)\left\{\int L_{\eta_{c_0}}d\mu_{c_t}-
			\min_{\mu\in \m} \int L_{\eta_{c_0}} d\mu\right\}
				+t\left\{\int L_{\eta_{c_1}}d\mu_{c_t}
				-\min_{\mu\in \m} \int L_{\eta_{c_1}} d\mu\right\}=0
		\end{align*}
		for all $t\in[0,1]$ and $\mu_{c_t}\in \m_{c_t}$ where $\eta_{c_i}$ is a closed 1-form on $M$ with
		$[\eta_{c_i}]=c_i$.
		Since $t\in [0,1]$ and
		\begin{align*}
			\int L_{\eta_{c_i}}d\mu_{c_t}
			- \min_{\mu\in \m} \int L_{\eta_{c_i}} d\mu\ge 0,\quad i=0,1,
		\end{align*}
		we have
		\begin{align*}
			\int L_{\eta_{c_i}}d\mu_{c_t}
			= \min_{\mu\in \m} \int L_{\eta_{c_i}} d\mu,\quad t\in(0,1),\quad i=0,1,
		\end{align*}
		which implies $\mu_{c_t}\in \m_{c_0}\cap \m_{c_1}$ for all $t\in(0,1)$.
		Thus (iii) holds.

		(ii)$\Rightarrow$(v):
		Assume that $\alpha$ is affine on $[c_0c_1]$.
		Let $t\in(0,1)$ and $(x,v)\in \A_{c_t}$.
		We denote $\gamma(s)=\pi(\Phi_s^L(x,v))$ for $s\in \Rset$
		where $\Phi_s^L$ is the Euler-Lagrange flow of $L$.
		Then, by the definition of $c_t$-static curves, we have
		\begin{align*}
			\int_a^b\left(L_{\eta_{c_t}}+\alpha(c_t)\right)_{(\gamma(s),\dot{\gamma}(s))}ds=
			-\phi_{\eta_{c_t}}(\gamma(b),\gamma(a))
		\end{align*}
		for all $a<b$. Since $\alpha$ is affine on $[c_0c_1]$,
		it follows that
		\begin{align*}
			(1-t)&\int_a^b\left(L_{\eta_{c_0}}+\alpha(c_0)\right)_{(\gamma(s),\dot{\gamma}(s))}ds+t\int_a^b\left(L_{\eta_{c_1}}+\alpha(c_1)\right)_{(\gamma(s),\dot{\gamma}(s))}ds\\
			&=
			-\inf_{T>0} \min_{\tilde{\gamma}\in C^T(\gamma(b),\gamma(a))}
			\Bigg\{(1-t)\int_0^T \left(L_{\eta_{c_0}}+\alpha(c_1)\right)_{(\tilde{\gamma}(s),\dot{\tilde{\gamma}}(s))}ds\\
			&\qquad\qquad\qquad\qquad\qquad\qquad\qquad
			+t\int_0^T \left(L_{\eta_{c_1}}+\alpha(c_1)\right)_{(\tilde{\gamma}(s),\dot{\tilde{\gamma}}(s))}ds\Bigg\}\\
			&\le
			-(1-t)\phi_{\eta_{c_0}}(\gamma(b),\gamma(a))
			-t\phi_{\eta_{c_1}}(\gamma(b),\gamma(a))
		\end{align*}
		for all $a<b$. Note that the above inequalities does not depend on the choices of representatives of $c_0$ and $c_1$.
		
		On the other hand, by \eqref{eqn:mane}, we have
		\begin{align*}
			\int_a^b\left(L_{\eta_{c_i}}+\alpha(c_i)\right)_{(\gamma(s),\dot{\gamma}(s))}ds
			\ge \phi_{\eta_{c_i}}(\gamma(a),\gamma(b))\\
			\ge -\phi_{\eta_{c_i}}(\gamma(b),\gamma(a))
		\end{align*}
		for all $a<b$ and $i=0,1$.
		Thus $\gamma$ is $c_i$-static curve for $i=0,1$.
		So we obtain $(x,v)\in \A_{c_0}\cap \A_{c_1}$.

		(vi)$\Rightarrow$(i):
		The main idea is based on the proof of Proposition 13 in \cite{A11}. Assume that $\A_{c_0}\cap\A_{c_1}\neq\emptyset$.
			Since $\A_{c_0}\cap\A_{c_1}$ is a nonempty and compact invariant set,
			using the Krylov-Bogolioubov theorem (\cite{W82}, Corollary 6.9.1), there exists an invariant probability measure $\mu$
			which satisfies
			\[
				\supp\ \mu\subset \A_{c_0}\cap\A_{c_1}.
			\]
			By Theorem~IV of \cite{CDI97}, it follows that
			\[
				\mu\in \m_{c_0}\cap\m_{c_1},
			\]
			which implies
			\[
				\int \big(L_{\eta_{c_i}}+\alpha(c_i)\big)d\mu=0,\quad i=0,1.
			\]
			For each $t\in[0,1]$, we have
			\[
				\int \big\{L_{\eta_{c_t}}+(1-t)\alpha(c_0)+t\alpha(c_1)\big\}d\mu=0.
			\]
			Thus,
			\[
				-(1-t)\alpha(c_0)-t\alpha(c_1)=\int L_{\eta_{c_t}}d\mu \ge -\alpha(c_t).
			\]
			Using the convex property of $\alpha$, we obtain the identity
			\[
				\alpha(c_t)=(1-t)\alpha(c_0)+t\alpha(c_1).
			\]
			So we have
			\[
				\int L_{c_t}d\mu=-\alpha(c_t),
			\]
			which means $\mu\in\m_{c_t}$.
			Moreover, since the inclusions \eqref{eqn:inclusions} hold, we see that
			\[
				E(\supp\ \mu)=E(\M_{c_t})=\alpha(c_t)
			\]
			for each $t\in[0,1]$. So $\alpha$ is constant on $[c_0 c_1]$.
	\end{proof}
	\begin{rmk}
		Proposition~\ref{prop:main} is an improvement of Proposition 3.3.6 (ii) of \cite{S15} and Proposition 6 in \cite{M03}. We remark that Proposition 3.3.6 (ii) of \cite{S15} asserts (i) $\Leftrightarrow$ (iii) and Proposition 6 in \cite{M03} does (vi) $\Rightarrow$ (i) and (ii) $\Rightarrow$ (2) respectively.
	\end{rmk}
	%

	Now we show Theorem~\ref{thm:exposed}.
	\begin{proof}[Proof of Theorem~\ref{thm:exposed}]
	From Proposition~\ref{prop:main}, we deduce that (ii), (iii) and (iv) are mutually equivalent. 
	Moreover, (v)$\Rightarrow$(iv) is trivial and the proof of (ii)$\Rightarrow$(v) is done similarly as in the part (vi)$\Rightarrow$(i) of Proposition~\ref{prop:main}. Note that $\supp \ \mu \subset \N_c$ if and only if $\mu\in\m_c$ for an invariant probability measure $\mu\in \m(L)$ and a cohomology class $c\in H^1(M;\Rset)$ (see Remark 4.1.27 and Proposition 4.1.28 in \cite{S15}).
	
		(i)$\Rightarrow$(ii): 
			Assume that (i) holds and $\alpha$ is affine on $[c c']$
			for some $c' \in H^1(M;\Rset)$.
			Then it follows from Proposition~\ref{prop:main} 
			that
			\[
				\M_{c_t}\subset\M_c\cap\M_{c'}
			\]
			for each $c_t:=(1-t)c+tc'$ where $t\in (0,1)$,
			which implies $\m_{c_t}\subset\m_c\cap\m_{c'}$.
			Therefore, by Proposition~\ref{prop:rotation}, we have
			\[
				\partial\alpha(c_t)=\rho(\m_{c_t})\subset
				\rho(\m_c)\cap\rho(\m_{c'})=\partial\alpha(c)\cap\partial\alpha(c').
			\]
			Since $\partial\alpha(c_t)$ is not empty (see (ii) of Proposition~\ref{prop:convex}),
			we can take $h\in\partial\alpha(c_t)$.
			Then $h$ also belongs to $\partial\alpha(c)$ and $\partial\alpha(c')$.
			Moreover, since (i) holds, $\beta$ is differentiable at $h\in \partial\alpha(c_t)\subset \partial\alpha(c)$.
			Thus,
			\[
				c_t, c, c'\in\partial \beta(h)=\{\nabla\beta(h)\},
			\]
			i.e., $c_t=c=c'=\nabla\beta(h)$. However, $c\neq c'$. This is contradiction.

		(ii)$\Rightarrow$(i): 
			Assume that (ii) holds and
			there exists $h\in \partial\alpha(c)$ such that $\#\partial\beta(h)\neq1$.
			Then there exists $c'\in H^1(M;\Rset)\setminus\{c\}$ such that $c'\in\partial\beta(h)$.
			So $h\in\partial\alpha(c)\cap\partial\alpha(c')$ holds.
			This implies
			\[
				\M^h\subset\M_{c}\cap\M_{c'}.
			\]
			Note that $h\in\partial\alpha(c)$ if and only if $\M^h\subset\M_{c}$ (see Lemma 3.5 in \cite{FGS09}).
			Since $\M^h$ is not empty, $\alpha$ is affine on $[cc']$ by Proposition~\ref{prop:main}. This is contradiction.
	\end{proof}

	\begin{proof}[Proof of Corollary~\ref{cor:C1-str}]
	(ii)$\Rightarrow$(i) is trivial and the converse immediately follows from (iii) of Proposition~\ref{prop:convex}.
	Since
	\begin{align}\label{eqn:mcmh}
		\bigcup_{c\in H^1(M;\Rset)}\m_c=\bigcup_{h\in H_1(M;\Rset)}\m^h
	\end{align}
	holds (see Equation (VI) in \cite{M96}), using Proposition~\ref{prop:rotation}, we have
	\[
		\bigcup_{c\in H^1(M;\Rset)}\partial\alpha(c)=\bigcup_{c\in H^1(M;\Rset)}\rho(\m_c)
		=\bigcup_{h\in H_1(M;\Rset)}\rho(\m^h)=H_1(M;\Rset).
	\]
	Hence Theorem~\ref{thm:exposed} implies that (i), (iv), (v), (vi) and (vii) are mutually equivalent.
	Moreover, Proposition~\ref{prop:main} implies (iii) $\Leftrightarrow$ (iv), which completes the proof.
	\end{proof}

	For Corollary~\ref{cor:C1-str-2}, we need the following lemma.
	\begin{lem}\
		\label{lem:map}
		\begin{itemize}
			\item[(i)]	If $\alpha$ is differentiable everywhere, then $\nabla\alpha:H^1(M;\Rset)\to H_1(M;\Rset)$ is a continuous and surjective map.
			\item[(ii)]	If $\beta$ is differentiable everywhere, then $\nabla\beta:H_1(M;\Rset)\to H^1(M;\Rset)$ is a continuous and surjective map.
		\end{itemize}
	\end{lem}
	\begin{proof}
		First, as in the proof of Corollary~\ref{cor:C1-str},
		\[
			\nabla\alpha(H^1(M;\Rset))= \rho(\m_{H^1(M;\Rset)})
			=\rho(\m^{H_1(M;\Rset)})=H_1(M;\Rset).
		\]
		Thus $\nabla\alpha$ is surjective.
		
		Next, pick $c\in H^1(M;\Rset)$. By Eq.~\eqref{eqn:mcmh},
		there exists $h\in H^1(M;\Rset)$ and $\mu\in\m^c$ such that $\rho(\mu)=h$.
		So
		\[
			h=\rho(\mu)\in\rho(\m_c)=\partial\alpha(c)
		\]
		holds by Proposition~\ref{prop:rotation}. Thus, we have
		\[
			c\in\partial\beta(h)=\{\nabla\beta(h)\},
		\]
		which implies that $\nabla\beta$ is surjective.
		
		The continuities of $\nabla\alpha$ and $\nabla\beta$ are consequences from Proposition~\ref{prop:convex} (iii).
	\end{proof}
	\begin{rmk}\label{rmk:homeo}
		By Lemma~\ref{lem:map} and Proposition~\ref{prop:convex} (iv) and (v), $\nabla\alpha$ (resp. $\nabla \beta$) becomes a bijection only  when both of $\alpha$ and $\beta$ are differentiable everywhere.
	\end{rmk}

	\begin{proof}[Proof of Corollary~\ref{cor:C1-str-2}]\
		(i)$\Leftrightarrow$(ii):
			 This immediately follows from Proposition~\ref{prop:convex} (iv) and Corollary~\ref{cor:C1-str}.\\
		(2): By Lemma~\ref{lem:map} and Remark~\ref{rmk:homeo}, we see that $\nabla\alpha$ and $\nabla\beta$ are homeomorphisms.
		Consider $c\in\{\nabla\beta(h)\}$. Then $h\in \{\nabla\alpha(c)\}$ holds and we have $c\in\{\nabla\beta(\nabla\alpha(c))\}$, which implies $\nabla\beta(\nabla\alpha(c))=c$ for each $c\in H^1(M;\Rset)$.
		Similarly, we have $\nabla\alpha(\nabla\beta(h))=h$ for each $h\in H_1(M;\Rset)$.
		Thus we deduce that $\nabla\beta=(\nabla\alpha)^{-1}$.\\
		(3): Since $\M_c=\cup_{h\in\partial\alpha} \M^h$ (see Equation (VIa) in \cite{M96}), we have $\M_c=\M^{\nabla \alpha (c)}$.
		Using $\nabla\beta=(\nabla\alpha)^{-1}$, we also have $\M^h=\M_{\nabla \beta (h)}$ for each $h\in H_1(M;\Rset)$.\\
		(1) and (4): These properties are consequences of (3) and Corollary~\ref{cor:C1-str}.
	\end{proof}

\section{Proofs of Theorems~\ref{thm:ext-alpha} and \ref{thm:C0}}
	
	\begin{proof}[Proof of Theorem~\ref{thm:ext-alpha}]
		(i): Assume that $c$ is not an extreme point of $\alpha$, i.e., there exists a segment $[c_0c_1]$ containing $c$ in its interior
		 such that $\alpha$ is affine on $[c_0c_1]$. Note that $c_0\neq c_1$.

		It follows that $\A_c\subset \A_{c_i}\ (i=0,1)$ by Proposition~\ref{prop:main}.
		Since $\pi(\A_c)=M$, using Mather's graph theorem,
		we have $\A_c= \A_{c_i}$ and $\pi(\A_{c_i})=M$ for $i=0,1$.
		By Eq.~\eqref{eqn:Fathi1} of Fathi's weak KAM theory,
		the Aubry set $ \A_{c_i}\ (i=0,1)$ is of the form
		\begin{align*}
			\A_{c_i}=\{(x,\eta_{c_i}(x)+du_{c_i}(x); x\in M\}
		\end{align*}
		for some closed 1-form $\eta_{c_i}$ with $[\eta_{c_i}]=c_i$
		and $C^1$-function $u_{c_i}:M\to \Rset$.
		The identity $\A_{c_0}=\A_{c_1}$ implies that
		\begin{align*}
			du_{c_0}(x)+\eta_{c_0}(x)=du_{c_1}(x)+\eta_{c_1}(x),\quad x\in M.
		\end{align*}
		Thus we deduce that
		\begin{align*}
			[\eta_{c_0}-\eta_{c_1}]=[d(u_{c_1}-u_{c_0})]=[0],
		\end{align*}
		which means $c_0=c_1$. This is contradiction.
		
		(ii): Assume that $\pi(\M_c)=M$ and $\alpha$ is differentiable at $c$.
		By the graph property of $\M_c$ and Eq.~\eqref{eqn:inclusions}, we have $\M_c=\A_c$.
		Moreover, $\pi(\A_c)=M$ and Remark~\ref{rmk:Aubry} imply that it is an invariant Lipschitz Lagrangian graph of the cohomology class $c$. 
		By the definition of the Mather set and Remark 3.1.11 in \cite{S15},
		there exists $\mu^* \in\m(L)$ such that ${\rm supp}\ \mu^*=\M_c$.
		Note that $\mu^*\in\m_c$ holds by Theorem~IV of \cite{CDI97}.
		Let $h^*=\rho(\mu^*)$ be the rotation vector of $\mu^*$.
		Since $\M_c={\rm supp} \mu^*\subset \M^{h^*}$ holds,
		by the graph property of $\M_c$ and $\pi(\M_c)=M$, we see that $\M_c=\M^{h^*}$ and $\pi(\M^{h^*})=M$.
		By Proposition~\ref{prop:rotation} and differentiability of $\alpha$ at $c$,
		we see that
		\[
			\rho(\m_c)=\partial\alpha(c)=\nabla\alpha(c).
		\]
		Therefore, $h^*=\nabla\alpha(c)$ holds.
		Consider $c' \in \partial\beta(h^*)$.
		Then $\M^{h^*}\subset \M_{c'}$ holds (see Lemma 3.5 in \cite{FGS09}). 
		Again by the graph property of $\M^{h^*}$ and $\pi(\M^{h^*})=M$, we obtain $\M_c=\M^{h^*}=\M_{c'}$.
		Since $\M_c$ is an invariant Lipschitz Lagrangian graph of the cohomology class $c$,
		we deduce that $c=c'$. Therefore, $\beta$ is differentiable at $h^*=\nabla\alpha(c)$.
		Using Theorem~\ref{thm:exposed}, we conclude that $c$ is an exposed point of $\alpha$.
	\end{proof}

	\begin{proof}[Proof of Theorem~\ref{thm:C0}]
		We borrow some techniques developed in \cite{M11}(Lemma 5, Proposition 4 and Theorem 3) and \cite{A15}(Proposition 5).\\
		(i): Let $H$ be a $C^0$ integrable Tonelli Hamiltonian
		and let $\Lambda_c$ be the invariant Lipschitz Lagrangian graph
		of cohomology class $c\in H^1(M;\Rset)$.
		By the graph property of $\Lambda_c$, for each $x\in M$, we can consider a map
		\begin{align*}
			F_x: c&\in H^1(M;\Rset)\mapsto v_c\in T^*_xM,
		\end{align*}
		where $\{v_c\}=\Lambda_c \cap T^*_xM$.
		Since the graphs $\{\Lambda_c\}_{c\in H^1(M;\Rset)}$ continuously foliate $T^*M$,
		the map $F_x$ is bijective and continuous.
		By the invariance of the Lagrange graph $\Lambda_c$, we see that $\Lambda_c$ is the graph of a solution of Hamilton-Jacobi equation $ H(x, \eta_c+du)=\alpha(c)$. Hence $H(\Lambda_c)=\alpha(c)$ holds.
		Now we show that $F_x$ is a closed map.
		For a compact subset $K\subset T_xM$,
		the set $F_x^{-1}(K)$ is closed subset in $H^{1}(M;\Rset)$.
		Moreover, for all $c\in F_x^{-1}(K)$, we have
		$\alpha(c)=H(\Lambda_c)=H(F_x(c))\in H(K)$.
		Since $K$ is compact and $H$ is continuous, $H(K)$ is a compact subset in $\Rset$.
		Thus, by the superlinearity of $\alpha$, $F_x^{-1}(K)$ is bounded.
		This implies that $F_x$ is proper.
		Since proper maps to locally compact spaces are closed, $F_x$ is a closed map.
		Consequently, $F_x$ is a bijective, continuous and closed map.
		Thus we deduce that $F_x$ is a homeomorphism.
		By the topological invariance of dimension, we have
		\[
			\dim H^1(M;\Rset)=\dim T^*_xM,
		\]
		which yields $\dim H^1(M;\Rset)=\dim M$.

		(ii):
		First we prove (1)$\Rightarrow$(2). Assume that $H$ is $C^0$ integrable.
		Then there exists a family of invariant Lipschitz Lagrangian graphs $\{\Lambda_c\}_{c\in H^1(M;\Rset)}$ and for each $c=[\eta_c]\in H^1(M;\Rset)$ the invariant Lipschitz Lagrangian graph $\Lambda_c$ is of the form $\Lambda_c=\{(x,\eta_c(x)+du_c(x));\ x\in M\}$ where $u$ is $C^{1,1}$-function on $M$.
		This implies $u$ is a weak KAM solution of $H(x,\eta_c+du_c)=\alpha(c)$.
	Hence we have $\Lambda_c\subset \N^*_c$ by Eq.~\eqref{eqn:Fathi2}.
		Since the cotangent bundle $T^*M$ is foliated by $\{\Lambda_c\}_{c\in H^1(M;\Rset)}$,
		we obtain
		\begin{align*}
			T^*M\quad=\bigcup_{c\in H^1(M;\Rset)} \Lambda_c\quad
			\subset\bigcup_{c\in H^1(M;\Rset)} \N_c.
		\end{align*}
		Thus the dual tiered Ma\~{n}\'{e} set
		\[
			\N^T_*=\bigcup_{c\in H^1(M;\Rset)} \N^*_c
		\] is the whole cotangent bundle $T^*M$.
		From the discussion in the proof of Theorem 1 (or Proposition 12) in \cite{A11},
		it follows that $\A^*_c=\N^*_c$ and it is the graph of a Lipschitz closed 1-form.
		Therefore, we see that $\pi(\A^*_c)=M$.

		Next we prove (2)$\Rightarrow$(1).
		Assume that (2) holds. By Remark~\ref{rmk:Aubry}, it follows that $\A^*_c$ is an invariant Lipschitz Lagrangian graph on $M$. Thus, as in the proof of (i), we can consider a map
		\begin{align*}
			F_x: c&\in H^1(M;\Rset)\mapsto v_c\in T^*_xM,
		\end{align*}
		where $\{v_c\}=\A^*_c \cap T^*_xM$. We now show that $F_x$ is a homeomorphism.
		By Theorem~\ref{thm:ext-alpha} (i), we see that $\alpha$ is strictly convex.
		Therefore, using Corollary~\ref{cor:C1-str}, we have $\A_c\cap\A_{c'}=\emptyset$ for any $c\neq c'\in H^1(M;\Rset)$, which implies that $F_x$ is injective.
		The rest of the proof is done similarly as in Proposition 5 in \cite{A15}.
	\end{proof}
	
\section{Examples}\label{sec:examples}
Finally, we investigate two examples: single-degree-of-freedom mechanical Hamiltonians and KAM tori.
	\subsection{Single-degree-of-freedom mechanical Hamiltonians}\label{subsec:1-degree}
	We first consider a single-degree-of-freedom mechanical Hamiltonian
	\begin{align*}
		H(x,p)=\frac{p^2}{2}+U(x)
	\end{align*}
	where $H:\Sset\times\Rset\to \Rset$ and $U:\Sset\to\Rset$ are $C^2$ functions.
	Note that the corresponding Lagrangian $L:\Sset\times\Rset\to \Rset$ is given by $L(x,v)=v^2/2-U(x)$.
	
	Let $c^*:=\int_0^1 \sqrt{2(\max U-U(x))}\ dx$. First, we consider the case $c^*\neq 0$.
	Let
	\begin{align*}
		C:=\{x\in \Sset; U(x)=\max U\},\quad \widetilde{C}=\{(x,0); x\in C\}.
	\end{align*}
	Note that $C\subsetneq M$ by $c^*>0$.
	For $E\ge\max U$, we set
	\begin{align*}
		c^\pm(E):=\pm\int_0^1 \sqrt{2(E-U(x))}dx.
	\end{align*}
	Then $c^+(E)$ (resp. $c^-(E)$) is strictly increasing (resp. decreasing) on $E$. Therefore, we can define its inverse $E^+(c):[c^*,+\infty)\to \Rset$ (resp. $E^-(c):(-\infty, -c^*] \to \Rset$) and it is also strictly increasing (resp. decreasing).
	Moreover, for $E\ge \max U$, let
	\begin{align*}
		\mathcal{P}^\pm_E=\left\{(x,v); v=\pm\sqrt{2(E-U(x))}, x\in \Sset\right\}.
	\end{align*}
	By the discussions of Sections 3.5 and 4.3 in \cite{S15}, we see that Mather's alpha function of $H$ is 
	\begin{align*}
		\alpha(c)=
		\begin{cases}
				0 & \mbox{for $c\in[-c^*,c^*]$};\\
				E^+(c) \quad & \mbox{for $c>c^*$};\\
				E^-(c) \quad & \mbox{for $c<-c^*$},
		\end{cases}
	\end{align*}
	and the Mather/Aubry sets are
	\begin{align*}
		\begin{cases}
				\M_c\quad=\quad \A_c\quad=\quad\widetilde{C} & \mbox{for $c\in[-c^*,c^*]$};\\
				\M_{\pm c}=\widetilde{C}\quad \subsetneq \quad \A_{\pm c}=\mathcal{P}^\pm_{\alpha(c)} \quad & \mbox{for $c=c^*$};\\
				\M_{\pm c} \quad=\quad \A_{\pm c}\quad=\quad\mathcal{P}^\pm_{\alpha(c)}\quad & \mbox{for $c>c^*$}.\\
		\end{cases}
	\end{align*}
	Note that $\alpha$ is $C^1$ since $\rho(\m_c)$ is a singleton for each $c\in\Rset$.
	We see that $\pi(\A_c)=M$ for the case $c=\pm c^*$ and
	$\pi(\M_c)=M$ for the case $|c|> c^*$.
	Moreover, $\A_c$ is a Lipschitz but not $C^1$ invariant Lagrange graph on $\Sset$ when $c=\pm c^*$
	and $\M_c$ is a $C^1$ invariant Lagrange graph on $\Sset$ when $|c|>c^*$.
	We now show that $c$ is an extreme but not exposed point of $\alpha$ if $c=\pm c^*$ and
	that $c$ is an exposed point of $\alpha$ if $|c|>c^*$. Since $\alpha(-c)=\alpha(c)$, we discuss only the case $c>0$.
	Assume that there exists a flat segment $[c_0,c_1]\subset [c^*,+\infty)$, i.e., $\alpha$ is affine on $[c_0,c_1]$. It follows from Proposition~\ref{prop:main} that $\alpha$ is constant on $[c_0,c_1]$, which contradicts strict monotonicity of $E^\pm(c)$.	
	
	Next, we consider the case $c^*=0$. In this case, $U(x)=\max U$ holds for all $x\in \Sset$, i.e., $U(x)$ is constant. Then $H(x,p)=p^2/2+\max U$.
	Moreover, we have $\alpha(c)=c^2/2+\max U$ and thus $\alpha$ is strictly convex and $C^1$.
	
	Summarizing, for a single-degree-of-freedom mechanical Hamiltonian of the form
		$H(x,p)=p^2/2+U(x)$ where $U$ is a $C^2$ potential function on $\Sset$,
		letting
		\begin{align*}
			c^*:=\int_0 ^1 \sqrt{2(\max U-U(x))}\ dx,
		\end{align*}
		we have: $\alpha$ is a $C^1$ function and
		\begin{itemize}
			\item[(i)] $c$ is a flat point of $\alpha$ if $|c|<c^*$. In this case, $\pi(\A_c)\subsetneq M$ holds and
			the Hamilton-Jacobi equation~\eqref{eqn:HJ} has
			a weak KAM solution which is Lipschitz but not $C^1$.
			\item[(ii)] $c$ is an extreme but not exposed point of $\alpha$ if $|c|=c^*\neq 0$.
			In this case, $\pi(\M_c)\subsetneq \pi(\A_c)=M$ holds and
			the Hamilton-Jacobi equation~\eqref{eqn:HJ} has
			a unique weak KAM solution (up to constants) which is $C^1$ but not $C^2$.
			\item[(iii)] $c$ is an exposed point of $\alpha$ if $|c|>c^*$.
			In this case, $\pi(\M_c)=M$ holds and
			the Hamilton-Jacobi equation~\eqref{eqn:HJ} has
			a unique weak KAM solution (up to constants) which is $C^2$.
		\end{itemize}
		Moreover, $\alpha$ is strictly convex if and only if $c^*=0$ if and only if $H(x,p)=p^2/2+\text{constant}$.
	
	\subsection{KAM tori}
	We next consider KAM tori. Let $H:\Tset^n\times\Rset^n\to \Rset$ be a Tonelli Hamiltonian.
	Following \cite{S15}, 
		$\mathcal{T}\subset \Tset^n\times\Rset^n$ is said to be
		a (maximal) KAM torus with rotation vector $h$ if
		\begin{itemize}
			\item[(1)]  $\mathcal{T}\subset \Tset^n\times\Rset^n$ is a $C^1$ Lagrangian graph over $\Tset^n$, i.e., $\mathcal{T}=\{(x, c+du); x\in\Tset^n\}$ where $c\in \Rset^n$ and $u\in C^2(\Tset^n,\Rset)$. 
			\item[(2)] $\mathcal{T}$ is invariant under the Hamiltonian flow $\Phi^H_t$ generated by $H$.
			\item[(3)] The Hamiltonian flow on $\mathcal{T}$ is conjugated to a uniform rotation on $\Tset^n$.
		\end{itemize}
	Then $u$ is a $C^2$ solution to the Hamilton-Jacobi equation~\eqref{eqn:HJ}.
	For the sake of simplicity, henceforth we assume that $h$ is rationally independent, as assumed in the original KAM theorem.
	
	Let $\mu^*$ be an ergodic invariant probability measure supported on a KAM torus
	$\mathcal{T}=\{(x, c+du); x\in\Tset^n\}$.
	By the rational independence of $h$, the orbit on $\mathcal{T}$ is dense in it and $\supp \ \mu^*=\mathcal{T}$ holds. Thus we have $\M_c=\mathcal{T}$ and $\pi(\M_c)=\Tset^n$.
	Since the Hamiltonian flow on $\mathcal{T}$ is conjugated to a uniform rotation on $\Tset^n$ and $\M_c=\mathcal{T}$ holds,
	all $\mu\in\m_c$ has the same rotation vector $h$.
	Hence, it follows from Proposition~\ref{prop:rotation} that Mather's alpha function of $H$ is differentiable at $c$.
	Therefore, by Theorem~\ref{thm:ext-alpha} (ii), we see that $c$ is an exposed point of Mather's alpha function of $H$.
	
	Summarizing, for a Tonelli Hamiltonian $H$ having a KAM torus $\mathcal{T}=\{(x, c+du); x\in\Tset^n\}$ with a rationally independent rotation vector $h$, we have:
	\begin{itemize}
		\item[(i)]  The Hamilton-Jacobi equation~\eqref{eqn:HJ} has a $C^2$ solution.
		\item[(ii)] $\pi(\M_c)=\Tset^n$ holds.
		\item[(iii)]  $c$ is an exposed and differentiable point of Mather's alpha function of $H$.
	\end{itemize}

\section*{Acknowledgement}
The author thanks Professors Daniel Massart and Alfonso Sorrentino for helpful comments on their results in \cite{M11}.
 





\end{document}